\documentclass{article}
\usepackage[utf8]{inputenc}

\usepackage{amsmath,amsfonts,amssymb,amsthm}
\usepackage[utf8]{inputenc}
\usepackage[T1]{fontenc}    
\usepackage{hyperref}       
\usepackage{url}            
\usepackage{booktabs}       
\usepackage{amsfonts}       
\usepackage{nicefrac}       
\usepackage{microtype}      
\usepackage{xcolor}         

\usepackage[nottoc,notlof,notlot]{tocbibind} 

\makeatletter
\let\l@section\l@chapter
\makeatother
\usepackage{amssymb}
\usepackage{amsthm}
\usepackage{hyperref} 
\usepackage[toc,page]{appendix}
\usepackage{algorithm}
\usepackage{algorithmic}
\usepackage{url}
\usepackage{graphicx}
\usepackage[many]{tcolorbox}
\usepackage{amsthm}
\usepackage{xcolor}
\usepackage{float}
\floatstyle{boxed} 
\usepackage{thm-restate}
\restylefloat{figure}
\usepackage{upgreek}
\usepackage{authblk}
\newtheorem{theorem}{Theorem}
\newtheorem{Definition}[theorem]{Definition}
\newtheorem*{remark}{Remark}

\newtheorem{lemma}[theorem]{Lemma}
\newtheorem{Proposition}[theorem]{Proposition}

\DeclareMathOperator{\Tr}{Tr}

\title{A geodesic convexity-like structure for the polar decomposition of a square matrix}

\author{Foivos Alimisis \hspace{4mm} Bart Vandereycken}
\vspace{2mm}
\affil{University of Geneva, Switzerland}
\date{}

\begin{document}

\maketitle

\begin{abstract}
We make a full landscape analysis of the (generally non-convex) orthogonal Procrustes problem. This problem is equivalent to computing the polar factor of a square matrix. We reveal a convexity-like structure, which explains the already established tractability of the problem and show that gradient descent in the orthogonal group computes the polar factor of a square matrix with linear convergence rate if the matrix is invertible and with an algebraic one if the matrix is singular. These results are similar to the ones of \cite{alimisis2024geodesic} for the symmetric eigenvalue problem. 
\end{abstract}

\section{Introduction}
\label{sec:intro}

The polar decomposition of a matrix is a standard factorization, where some matrix $C \in \mathbb{R}^{m \times n}$, $m \geq n$, must be written as the product of an orthonormal matrix $X \in \mathbb{R}^{n \times m}$ and a symmetric and positive semidefinite matrix $P \in \mathbb{R}^{n \times n}$, i.e.
\begin{equation*}
    C=X P.
\end{equation*}

Such a decomposition always exists and a good way to see that is through the singular value decomposition. If a singular value decomposition of $C$ is
\begin{equation*}
    C=U \Sigma V^T,
\end{equation*}
then the "polar factor" $X$ of the polar decomposition is given as
\begin{equation*}
    X=U V^T
\end{equation*}
and the symmetric positive semidefinite part $P$ is given as
\begin{equation*}
    P=V \Sigma V^T.
\end{equation*}

One can easily see that the polar decomposition of $C$ is unique if and only if $C$ is invertible, i.e. if and only if its singular values are all positive.  

An interesting property of the polar factor is that it is the closest orthonormal matrix to the original matrix $C$ (see \cite{higham2008functions}, Theorem 8.4). This makes polar decomposition intimately related to the orthogonal Procrustes problem (see \cite{higham2008functions}, Theorem 8.6). The Procrustes problem \cite{schonemann1966generalized} is important in many areas of applied science \cite{akca2003generalized,faigenbaum2024studying,kabsch1976solution}. It seeks for an orthogonal matrix $X \in \mathbb{O}(n)$, such that the quantity $\|AX-B\|_F^2$ for two matrices $A,B \in \mathbb{R}^{m \times n}$ is as small as possible. This problem admits the equivalent formulation
\begin{equation*}
    \min_{X \in \mathbb{O}(n)} -\Tr(CX),
\end{equation*}
with $C:=B^T A$,
and its solution is the polar factor of the matrix $C^T \in \mathbb{R}^{n \times n}$. This problem turns out to have a geodesic convexity-like structure in the orthogonal group, which we analyze in Section \ref{sec:main_theory}. This structure is similar to the one that has recently been proven for the symmetric eigenvalue problem \cite{alimisis2024geodesic}. Using this convexity-like structure, we analyze a Riemannian gradient descent algorithm in the orthogonal group for computing the polar factor of $C$. This algorithm is in general slow and expensive compared to the state-of-the-art and is presented only for theoretical purposes.

\section{Related work}
\label{sec:related_work}
The most direct way to compute a polar decomposition is via the SVD. Clearly, this approach is too expensive.
The numerical linear algebra community has developed plenty of faster algorithms to tackle the problem of computing a polar factor. The most basic one is the Newton method (\cite{higham2008functions}, Section 8.3). The Newton method is in general fast in the late stage of convergence, but can be very slow in the beginning if the matrix $C$ is ill-conditioned. Another prominent class of algorithms is the Padé family of iterations (\cite{higham2008functions}, Section 8.5), which suffers more or less by the same issues.

Most of the effort in the last few years has been focused on scaling the basic Newton iteration, in order to obtain variants that do not suffer from slow convergence in the beginning of the iterations. The so-called "optimal" scaling \cite{4791187} enjoys excellent theoretical behaviour, but the scaling factor depends on the (generally unknown) smallest and largest singular values in each iterate $X_t$. A more practical version, that however lacks convergence guarantees, can be found in \cite{higham1986computing}. A middle ground with a sub-optimal computable scaling that still enjoys some convergence guarantees can be found in \cite{doi:10.1137/070699895}.

The state-of-the art in this area comes probably from \cite{nakatsukasa2010optimizing}. There, the Halley's method (which is a member of the Padè family of iterations) is scaled in a principled way. The Halley method has cubic asymptotic convergence, but the initial stage can be very slow for ill-conditioned matrices \cite{5d07223e-70da-35b8-b26a-e407951886aa}. The scaling of \cite{nakatsukasa2010optimizing} helps to improve its performance in the initial stage of convergence.

Other studies on meta-issues like the backward stability of the aforementioned algorithms have also been conducted \cite{nakatsukasa2012backward}, but this starts deviating from the purpose of our work.

Our work shows that the problem of computing the polar factor of a square matrix can be treated in the realm of convex optimization. That is because this problem enjoys a structure that we call weak-quasi-strong-convexity (WQSC) (Proposition \ref{prop:wqsc}). This property has become a staple in optimization in recent years and versions of it have appeared in \cite{necoara2019linear} (Definition 1), \cite{karimi2016linear}(Appendix A) and \cite{bu2020note}, and it is equivalent with weak-quasi-convexity \cite{guminov2017accelerated,hardt2018gradient} and the more well-known quadratic growth condition \cite{drusvyatskiy2018error} holding simultaneously. WQSC has been derived also for the symmetric eigenvalue problem in \cite{alimisis2024geodesic}. It guarantees a linear convergence rate for gradient descent in terms of distances of the iterates to the optimum in the case that $C$ is non-singular (Theorem \ref{thm:GD_conv}). Actually, WQSC has been proven to be also a necessary property for having this kind of convergence \cite{alimisis2024characterization}, bearing a special meaning for optimization in general.

\section{Geometry of the orthogonal group}

We present here the basics of the geometry of our space of interest, i.e. the orthogonal group ($\mathbb{O}(n)$).
$\mathbb{O}(n)$ is a Lie group (thus also a manifold) consisted by orthogonal $n \times n$ matrices. It is consisted by two connected components, the orthogonal matrices with determinant $+1$ and the ones with determinant $-1$. We recall here basic aspects of the geometry of the orthogonal group. A more complete exposition can be found along the excellent text \cite{salamon2019}.

The tangent space at a point $X$ is
\begin{equation*}
    T_X \mathbb{O}(n) = \lbrace{X \Omega \ | \ \Omega \in \mathbb{R}^{n \times n} \hspace{1mm} \text{is skew-symmetric}} \rbrace.
\end{equation*}
\newline
The usual Riemannian metric one equips this space is just the Euclidean one:
\begin{equation*}
    \langle V , W \rangle_X= \Tr(W^T V).
\end{equation*}
\newline
The orthogonal projection of a matrix $Z \in \mathbb{R}^{n \times n}$ onto $T_X \mathbb{O}(n)$ is 
\begin{equation}
\label{eq:orth_proj}
    P_X (Z) = X \textnormal{skew}(X^T Z). 
\end{equation}
\newline
The exponential map at a point $X$ in the direction $X \Omega$ is defined as
\begin{equation*}
    \exp_X (X \Omega) = X \exp_m(\Omega),
\end{equation*}
where $\exp_m$ is the matrix exponential.
\newline
The Riemannian logarithm is the inverse of the exponential map, when the latter is invertible. We now examine when this is the case.
The domain where the exponential map is a diffeomorphsim is usually called injectivity domain. In order to identify it, we need to verify when the equation
\begin{equation*}
    \exp_X(X \Omega) = X \exp_m(\Omega) = Y
\end{equation*}
has a unique solution. This happens if and only if the equation
\begin{equation*}
\exp_m(\Omega) = X^T Y    
\end{equation*}
has a unique solution. Considering the eigenvalue decomposition $\Omega=U \Lambda U^{-1}$, where $\Lambda$ is diagonal with entries of the form $i r$ with $r \in (-\pi,\pi]$ (since $\Omega$ is skew-symmetric). This implies that the eigenvalue decomposition of $\exp_m(\Omega)$ is $U \exp_m(\Lambda) U^{-1}$ and $\exp_m(\Lambda)$ is diagonal featuring entries of the form $e^{i r}$. Thus, the previous equation is equivalent to the diagonal system $\exp_m(\Lambda) = U^{-1} X^T Y U$. Reading the diagonal entries, we obtain a series of equations of the form
\begin{equation*}
    e^{i r} = s,
\end{equation*}
where $s$ are the eigenvalues of $X^T Y$. These equations are well-defined and have a unique solution if and only if $s$ is in the domain of a definition of the complex logarithm, $\mathbb{C} \setminus (-\infty,0]$. In that case, $r$ is allowed to be in $(-\pi,\pi)$. We can summarize the previous discussion as follows:
\begin{lemma}
\label{lem:orth_group_misc}
    \begin{itemize}
        \item The domain of the orthogonal group where the exponential map is a diffeomorphishm is
        \begin{equation}
        \label{inj_domain}
          \lbrace X \Omega \ | \  \Omega^T = -\Omega, \|\Omega\|_2<\pi \rbrace.  
        \end{equation}
        \item Let $X,Y \in \mathbb{O}(n)$. If the phases $r$ of the eigenvalues $e^{i r}$ of $X^T Y$ satisfy $r \in (-\pi,\pi)$, then there is a unique geodesic connecting $X$ and $Y$. In this case, it trivially holds that $X$ and $Y$ are in the same connected component of $\mathbb{O}(n)$.
        \item If some of the $r$'s are equal to $\pi$, then it holds: if there is even number of $r$'s equal to $\pi$, then $X$ and $Y$ are in the same connected component (and are connected by multiple geodesics). If there is odd number of $r$'s equal to $\pi$, then $X$ and $Y$ are in different connected components (i.e. $\det(XY)=-1$).
     \end{itemize}
\end{lemma}

Let us now consider $X$ and $Y$ such that $X^T Y$ has eigenvalues with phases in $(-\pi,\pi)$. Then $\log_X(Y)$ is well-defined and 
\begin{equation*}
    \exp_X(\log_X (Y))=Y.
\end{equation*}
We can write $\log_X (Y)=X \Omega$ for some skew-symmetric $\Omega$ and we have
\begin{equation}
\label{eq:Riem_exponential}
    X \exp_m(\Omega) = Y,
\end{equation}
which can be written as
\begin{equation*}
    \Omega = \log_m(X^T Y),
\end{equation*}
where $\log_m$ is the matrix logarithm. 

Thus,
\begin{equation}
\label{eq:Riem_logarithm}
    \log_X (Y) = X \log_m(X^T Y).
\end{equation}
Note that $\log_m(X^T Y)$ is indeed a skew-symmetric matrix since $X$ and $Y$ are orthogonal.

$\log_X(Y)$ is a tangent direction at $X$ that allows to move from $X$ to $Y$
along the geodesic connecting them. In Riemannian manifolds, there is a general rule to transport tangent vectors from one point to another one parallely to the geodesics of the manifold, called parallel transport. In the orthogonal group, the parallel transport from a point $X$ to a point $Y$ (denoted by $\Gamma_X^Y$), is given by
\begin{equation*}
    \Gamma_X^Y (X \Omega) = Y (X^T Y \Omega Y^T X).  
\end{equation*}
Notice that $X^T Y \Omega Y^T X$ is a skew-symmetric matrix, since it is a conjugation of the slew-symmetric matrix $\Omega$. This definition makes sense of course only if $X$ and $Y$ are in the same connected component of $\mathbb{O}^n$.

Since we have computed the Riemannian logarithm between two orthogonal matrices $X$ and $Y$, we can also compute the Riemannian distance between such matrices based on it:
\begin{equation*}
    \textnormal{dist}^2(X,Y) = \|\log_X (Y)\|^2 = \|X \log_m(X^T Y) \|^2 = \| \log_m(X^T Y) \|^2.
\end{equation*}

In order to proceed, we decompose the orthogonal matrix $X^T Y$ into the so-called canonical form $P D P^T$, where $P$ is orthogonal featuring the eigenvectors of $X^T Y$ in its columns and $D$ is block diagonal. $D$ is constructed as follows. When $X^T Y$ has an eigenvalue equal to $1$, $D$ has a diagonal entry equal to $1$. When $X^T Y$ has an eigenvalue of the form $e^{i r}$ for some $r \in (-\pi,0) \cup (\pi,0) $, then $e^{-i r}$ is also an eigenvalue and $D$ features the $2 \times 2$ block that is the 2-d rotation with angles $r$. That is $\begin{bmatrix}
\cos r & -\sin r \\
\sin r & \cos r
\end{bmatrix}.
$

The matrix logarithm has the following convenient property. Given the above decomposition, we have
\begin{equation*}
    \log_m(P D P^T) = P \log_m(D) P^T.
\end{equation*}

Taking $D$ as constructed previously, $\log_m(D)$ has $0$ in the positions where $D$ has $1$ and
$\begin{bmatrix}
0 & -r \\
r & 0
\end{bmatrix}$
where $D$ has
$\begin{bmatrix}
\cos r & -\sin r \\
\sin r & \cos r
\end{bmatrix}$.
Since $P$ is orthogonal, the distance between $X$ and $Y$ turns out to be equal to $\| \log_m (D) \|^2 = \Tr(\log_m(D)^T \log_m(D))$. $\log_m(D)^T \log_m(D)$ is again a $2 \times 2$ block diagonal matrix with $0$'s where $\log_m(D)$ has $0$'s and $\begin{bmatrix}
r^2 & 0 \\
0 & r^2
\end{bmatrix}$
where $\log_m(D)$ has $\begin{bmatrix}
0 & -r \\
r & 0
\end{bmatrix}$. Thus, the distance between $X$ and $Y$ is
\begin{equation}
\label{eq:Riem_distance}
    \textnormal{dist}(X,Y) = \left(\sum_{i=1}^n r_i^2 \right)^{1/2},
\end{equation}
where $e^{i r_i}$ are the eigenvalues of $X^T Y$. That is to say that
\begin{equation*}
    \textnormal{dist}(X,Y) = \| \phi \|_2,
\end{equation*}
where $\phi=(r_1, \dots, r_n)$. If $r_j=0$, then it appears only once in $\phi$, otherwise it appears as a couple with $-r_j$. Note that with a simple limit argument, we can conclude that the same formula still holds when some eigenvalues of $X^T Y$ have a phase equal to $\pi$.

We end this section on the geometry basics of the orthogonal group by discussing its sectional curvatures. The only fact about its curvatures that will be useful to us is that they are always nonnegative. This is folklore as the sectional curvatures of all Stiefel manifolds are nonegative, thus also the ones of the orthogonal group (which is a special case of a Stiefel manifold). This gives rise to the following useful geometric bound:

\begin{Proposition}
\label{prop:law_of_cosines}
    Consider three points $X,Y,Z \in \mathbb{O}(n)$, such that they are connected by unique geodesics. Then, we have
    \begin{enumerate}
        \item $
    \textnormal{dist}^2(X, Y) \leq \textnormal{dist}^2(Z, X)+ \textnormal{dist}^2(Z, Y)-2 \langle \textnormal{Log}_{Z} (X), \textnormal{Log}_{Z} (Y) \rangle.
$
    \item $\textnormal{dist}(X,Y) \leq \| \log_Z(X) - \log_Z(Y) \| $.
    
    \end{enumerate}    
\end{Proposition}

\begin{proof}
     Both inequalities are a consequence of the famous Toponogov's theorem, taking into account that the sectional curvatures of $\mathbb{O}(n)$ are nonnegative.
\end{proof}

\section{Convexity-like properties of orthogonal Procrustes}
\label{sec:main_theory}
We investigate now thoroughly the orthogonal Procrustes problem. This problem concerns with finding orthogonal matrices $X_1$ and $X_2$ that best fit two other matrices $A,B \in \mathbb{R}^{k \times n}$:
\begin{equation*}
    \min_{X_1,X_2 \in \mathbb{O}(n)} \|A X_1-B X_2 \|^2.
\end{equation*}

Since this problem is invariant under simultaneous right multiplication of $X_1$ and $X_2$ with an orthogonal matrix $Q \in \mathbb{R}^{n \times n}$, we can fix $X_2$ to be identity and target only the matrix $X_1 \rightsquigarrow X$:
\begin{equation*}
    \min_{X \in \mathbb{O}(n)} \| A X-B\|^2.
\end{equation*}
This problem can be written equivalently as
\begin{equation}
\label{main_problem}
    \min_{X \in \mathbb{O}(n)} -\Tr(CX)=:f(X),
\end{equation}
where 
\begin{equation*}
    C:= B^T A.
\end{equation*}

It has a global solution and can be found in closed form \cite{schonemann1966generalized}: if $C=U \Sigma V^T$ is an SVD of $C$, then a global solution is $X_* = V U^T$. The minimum $f_*:=f(X_*)$ is the opposite of the sum of the singular values of $C$.

We will use this structure to prove a quasi-convexity property for the function
\begin{equation*}
    f(X) = -\Tr(C X)
\end{equation*}
around $X_*$.

It is well known that the solution of the problem is \textbf{unique} if and only if all the singular values of $C$ are strictly positive, i.e. if and only if $C$ is invertible.

\paragraph{Riemannian gradient:} 
To compute the Riemannian gradient of $f$, we need just to project the Euclidean gradient $\nabla f(X) = -C^T$ onto the tangent space $T_X \mathbb{O}(n)$. This results to
\begin{equation}
\label{eq:riem_grad}
    \textnormal{grad}f(X) = P_X(-C^T) = -X \textnormal{skew}(X^T C^T).
\end{equation}

\paragraph{Riemannian Hessian:} 
For a function $f$ defined in the orthogonal group, we have (see \cite{boumal2023})
    \begin{equation*}
        D \textnormal{grad}f(X)[\dot X] = \dot X \textnormal{skew} (X^T \nabla f(X)) + X \textnormal{skew}(\dot X^T \nabla f(X)+\dot X^T \nabla^2 f(X)[\dot X]),
    \end{equation*}
where $\dot X = X \Omega$ is an arbitrary tangent vector. In our case, $\nabla f(X)=-C^T$ and $\nabla^2 f(X)=0$, thus

\begin{equation}
\label{eq:riem_hess}
    \textnormal{Hess} f(X) [\dot X] =  -\dot X \textnormal{skew} (X^T C^T) - X \textnormal{skew}(\dot X^T C^T).
\end{equation}

\begin{Proposition}[Geodesic weak-quasi-convexity]
\label{prop:wqc}
   Let $X_* \in \mathbb{O}(n)$ a global optimum of the function $f:\mathbb{O}(n) \rightarrow \mathbb{R}$. Let also $X \in \mathbb{O}(n)$ such that the eigenvalues $e^{i r}$ of $X^T X_*$ are such that $r \in (-\pi,\pi)$. If $|r|_{\max}$ denotes the largest possible rotation induced by $X^T X_*$ in absolute value, then
    \begin{equation*}
        \langle \textnormal{grad}f(X), -\log_{X}(X_*) \rangle \geq \frac{1}{2} (1+\cos (|r|_{\max})) (f(X)-f_*).
    \end{equation*}
\end{Proposition}

\begin{proof}

The Riemannian gradient of $f$ is given in equation \ref{eq:riem_grad}. It remains to compute a convenient expression for the Riemannian logarithm.

According to equation (\ref{eq:Riem_logarithm}), the Riemannian logarithm is given as
\begin{equation*}
    \log_X(X_*) = X \log_m(X^T X_*). 
\end{equation*}
As in the introduction, we use the canonical form of the orthogonal matrix $X^T X^*$
\begin{equation*}
    X^T X_* = P D P^T.
\end{equation*}

Since the matrix logarithm is invariant under conjugate action, we have
\begin{equation*}
    \log_m(X^T X_*) = P \log_m (D) P^T
\end{equation*}
and $\log_m(D)$ is again a block diagonal matrix, with blocks being the logarithms of the blocks of $D$: when $D$ has a diagonal entry equal to $1$, $\log_m(D)$ has a diagonal entry equal to $0$ and when $D$ features a $2 \times 2$ block, which is a rotation of angle $r$, $\log_m(D)$ features the block
$\begin{bmatrix}
0 & -r \\
r & 0.
\end{bmatrix}
$

Similarly, the skew-symmetric part of $X^T X_*$ satisfies
\begin{equation*}
    \textnormal{skew}(X^T X_*) = P \textnormal{skew}(D) P^T,
\end{equation*}
where $\textnormal{skew}(D)$ is again block diagonal and has a $0$ diagonal entry when $D$ has a $1$ diagonal entry, while it has a block $\begin{bmatrix}
0 & -\sin r \\
\sin r & 0
\end{bmatrix}
$
when $D$ features a $2 \times 2$ rotation of angle $r$. Thus, it holds in general that
\begin{equation*}
    \log_m(D) =  \textnormal{skew}(D) \frac{\phi}{\sin \phi}, 
\end{equation*}
where $\phi = (r_1,\dots,r_n)$ is a vector capturing all the rotations induced by the orthogonal matrix $X^T X_*$. If $r=0$, i.e. corresponds to a diagonal entry equal to $1$, then it appears only once in $\phi$, while if $r \in (-\pi,\pi) \setminus \lbrace 0 \rbrace$ it appears as a couple with $-r$. \newline
$\phi / \sin \phi$ is a diagonal matrix with diagonal elements $r_j/\sin r_j$. This convention is made for ease of notation. 

Given that, we can write
\begin{align*}
    &\log_X(X_*) = X \log_m(X^T X_*) = X P \log_m(D) P^T = X P \textnormal{skew}(D) \frac{\phi}{\sin \phi} P^T \\ & = X P \textnormal{skew}(D) P^T P \frac{\phi}{\sin \phi} P^T  = P_X(X_*) P \frac{\phi}{\sin \phi} P^T.
\end{align*}
Now we can finally deal with the desired inequality:
\begin{align*}
    & \langle \textnormal{grad}f(X) , -\log_X(X_*) \rangle = \left \langle P_X(C^T) , P_X(X_*) P \frac{\phi}{\sin \phi} P^T \right \rangle =  \\ &  \left \langle X \textnormal{skew}(X^T C^T) , X_* P \frac{\phi}{\sin \phi} P^T \right \rangle = \Tr \left(P \frac{\phi}{\sin \phi} P^T X_*^T X \textnormal{skew}(X^T C^T) \right ).
\end{align*}

We pause to deal with the term $X_*^T X \textnormal{skew}(X^T C^T)$:

\begin{align*}
    X_*^T X \textnormal{skew}(X^T C^T) = X_*^T X \frac{X^T C^T-CX}{2} = \frac{X_*^T C^T - X_*^T X CX}{2}.
\end{align*}

Remember that if $X_* = V U^T$, then $C=U \Sigma V^T$ is an SVD of $C$. Thus 
\begin{equation*}
    X_*^T C^T = U \Sigma U^T
\end{equation*}
and
\begin{equation*}
    X_*^T X CX = P D^T P^T U \Sigma V^T V U^T P D^T P^T = P D^T P^T U \Sigma U^T P D^T P^T.
\end{equation*}

Plugging this expression in, we get

\begin{align*}
    & 2 \langle \textnormal{grad}f(X) , -\log_X(X_*) \rangle = \Tr\left(P \frac{\phi}{\sin \phi} P^T X_*^T X \textnormal{skew}(X^T C^T) \right ) \\ & = \Tr \left(P \frac{\phi}{\sin \phi} P^T (U \Sigma U^T-P D^T P^T U \Sigma U^T P D^T P^T) \right) \\ & = \Tr \left( \frac{\phi}{\sin \phi} (P^T U \Sigma U^T P - D^T P^T U \Sigma U^T P D^T) \right ) \\ & = \Tr \left( \frac{\phi}{\sin \phi} (P^T U \Sigma U^T P)\right)-\Tr\left(\frac{\phi}{\sin \phi} (D^T P^T U \Sigma U^T P D^T) \right) \\ & =
    \Tr \left( \left(\frac{\phi}{\sin \phi} - D^T \frac{\phi}{\sin \phi} D^T \right) P^T U \Sigma U^T P \right)
    .
\end{align*}

It suffices to show that 
\begin{align*}
    & \Tr \left( \left(\frac{\phi}{\sin \phi} - D^T \frac{\phi}{\sin \phi} D^T \right) P^T U \Sigma U^T P \right) \geq  (1+\cos (|r|_{\max})) (f(X)-f_*) = \\ & \underbrace{(1+\cos (|r|_{\max}))}_{:=c} \left(\underbrace{\Tr(P^T U \Sigma U^T P)}_{-f_*} - \underbrace{\Tr(D^T P^T U \Sigma U^T P)}_{-f(X)}  \right).
\end{align*}

This holds if 
\begin{equation*}
    \Tr\left( \left(\frac{\phi}{\sin \phi} - D^T \frac{\phi}{\sin \phi} D^T +c (D^T - I) \right) \underbrace{P^T U \Sigma U^T P}_{:=A} \right) \geq 0.
\end{equation*}
Notice that the matrix $A$ is symmetric and positive semi-definite.

$D^T$ is a matrix with diagonal entries equal to $1$ and $2 \times 2$ diagonal blocks of the form $\begin{bmatrix}
\cos r & \sin r \\
-\sin r & \cos r
\end{bmatrix}
$, which essentially correspond to rotations with $-r$. Multiplying with the diagonal matrix $\phi/\sin \phi$ from the right, keeps the $1$ diagonal entries of $D^T$ unchanged, while it transforms the $2 \times 2$ diagonal blocks to 
$\begin{bmatrix}
r/\tan r & r \\
-r & r/\tan r
\end{bmatrix}$.
The matrix $D^T \frac{\phi}{\sin \phi} D^T$ still keeps $1$ in the entries that correspond to $r=0$ and has $2 \times 2$ diagonal blocks associated with $r \in (-\pi,\pi)\setminus \lbrace 0 \rbrace$ that are $\begin{bmatrix}
\frac{r}{\tan r} \cos r -r \sin r & \frac{r}{\tan r} \sin r + r \cos r\\
-r \cos r - \frac{r}{\tan r} \sin r & \frac{r}{\tan r} \cos r -r \sin r
\end{bmatrix}$.
\vspace{4mm}

The matrix $\frac{\phi}{\sin \phi} - D^T \frac{\phi}{\sin \phi} D^T + c(D^T-I)$  has $1$ in the diagonal entries that $D^T$ has $1$ ($r=0$) and has $2 \times 2$ diagonal blocks that correspond to rotations with $r \in (-\pi,\pi)\setminus \lbrace 0 \rbrace$, which are 
\begin{equation*}
    \begin{bmatrix}
\frac{r}{\sin r} - \frac{r}{\tan r} \cos r + r \sin r + c(\cos r - 1) &  -\frac{r}{\tan r} \sin r - r \cos r + c\sin r \\
\frac{r}{\tan r} \sin r + r \cos r-c\sin r & \frac{r}{\sin r} - \frac{r}{\tan r} \cos r + r \sin r + c(\cos r - 1)
\end{bmatrix}.
\end{equation*}

\vspace{4mm}

Notice that this last $2 \times 2$ matrix is of the form $\begin{bmatrix}
\alpha &  \beta \\
-\beta & \alpha
\end{bmatrix}.$

The expression $\Tr\left( \left(\frac{\phi}{\sin \phi} - D^T \frac{\phi}{\sin \phi} D^T +c (D^T - I) \right) A \right)$ that we want to prove nonnegative is the sum of the traces of the product of the diagonal entries of $\frac{\phi}{\sin \phi} - D^T \frac{\phi}{\sin \phi} D^T +c (D^T - I)$ that correspond to $r=0$ (i.e. $1$) with the corresponding diagonal entries of $A$ and the $2 \times 2$ diagonal blocks of $\frac{\phi}{\sin \phi} - D^T \frac{\phi}{\sin \phi} D^T +c (D^T - I)$ with the corresponding $2 \times 2$ diagonal blocks of $A$. In the first case we get back the diagonal entries of $A$ (which are nonnegative) and in the second case we have the product of a matrix of the form $\begin{bmatrix}
\alpha &  \beta \\
-\beta & \alpha
\end{bmatrix}$
with one of the form 
$\begin{bmatrix}
s & t \\
t & k 
\end{bmatrix} $, since $A$ is symmetric. The diagonal entries of this product (which are the ones that contribute in the trace) are
$\alpha s + \beta t$ and $ -\beta t + \alpha k $. Their sum is
$\alpha (s+t)$, thus it suffices to show that this expression is nonnegative, i.e. that $\alpha$ is nonnegative since $s$ and $t$ are nonnegative as diagonal entries of the positive semi-definite matrix $A$.

Remember that $\alpha$ has been taken as
\begin{align*}
    \alpha :&= \frac{r}{\sin r} - \frac{r}{\tan r} \cos r + r \sin r + \underbrace{(1+\cos (|r|_{\max}))}_{c}(\cos r - 1) \\ & \geq \frac{r}{\sin r} - \frac{r}{\tan r} \cos r + r \sin r + (1+\cos r)(\cos r - 1),
\end{align*}
since $ r  \leq | r|_{\max} $ and $\cos r -1 \leq 0$. The last lower bound for $\alpha$ turns out to be positive for all $r \in (-\pi,\pi)$, thus our proof is complete.

\end{proof}

We now examine a property for $f$ known as quadratic growth. This property gives a non-trivial inequality only in the case that the Procrustes problem has a unique solution (i.e. if and only if $C$ is non-singular).

\begin{Proposition}[Quadratic growth]
\label{prop:quad_growth}
Let $X_* \in \mathbb{O}(n)$ to be a global minimizer for $f$ and $X \in \mathbb{O}(n)$ in the same connected component. Then
    $f$ satisfies
    \begin{equation*}
        f(X)-f_* \geq \frac{2 \sigma_{min}(C)}{\pi^2} \textnormal{dist}^2 (X,X_*),
    \end{equation*}
    where $\sigma_{\min}(C)$ is the smallest singular value of $C$.
\end{Proposition}

\begin{proof}
Recall that if $C=U \Sigma V^T$ is an SVD of $C$, then
$X_* = V U^T$ is a global minimizer. Consider again the canonical form of the orthogonal matrix $X^T X_*$:
\begin{equation*}
    X^T X_* = P D P^T.
\end{equation*}

Then, we have

\begin{align*}
    f(X)-f_* &= -\Tr(CX)+\Tr(C X_*) = \Tr(P^T U \Sigma U^T P) - \Tr(U \Sigma U^T P D^T P^T) \\ & = \Tr((I-D^T) \underbrace{P^T U \Sigma U^T P}_{\text{pos. semi-definite}})
\end{align*}

Let us denote again $A:=P^T U \Sigma U^T P$, which is symmetric and positive semi-definite. 
The matrix $I-D^T$ has diagonal entries equal to $0$ for rotations $r = 0$, diagonal entries equal to $2$ for $r=\pi$ and $2 \times 2$ diagonal blocks of the form
$\begin{bmatrix}
1-\cos r & \sin r \\
-\sin r & 1-\cos r
\end{bmatrix} $ for rotations with angle $r \in (-\pi,\pi)$ not $0$. Thus, the diagonal entries of the product $(I-D^T) P^T U \Sigma U^T P$ are either $0$ for entries that correspond to no rotation, i.e. $(1-\cos r) A_{ii}$, or the diagonal entries of a product of the form
\begin{equation*}
   \begin{bmatrix}
1-\cos r & \sin r \\
-\sin r & 1-\cos r
\end{bmatrix} 
\begin{bmatrix}
s & t \\
t & k
\end{bmatrix}.
\end{equation*}
These are $(1-\cos r) s + \sin r t$ and $ - \sin r t + (1- \cos r) k$. Since summing them makes the terms $\sin r t$ to cancel out, we get
\begin{equation*}
    \Tr((I-D^T) P^T U \Sigma U^T P) = \Tr((I-\cos \phi) P^T U \Sigma U^T P),
\end{equation*}
where $\phi=(r_1,\dots,r_n)$ a vector capturing all the rotations between $X$ and $X_*$. If $r_j=0$ or $\pi$, then it appears only once, if $r_j \neq 0, \pi$ it appears coupled with its opposite $-r$. Notice that
\begin{equation*}
    \| \phi \| = \textnormal{dist}(X,X_*).
\end{equation*}

Since for all $r$ it holds $r \in (-\pi,\pi]$, we have
\begin{equation*}
    1-\cos r \geq \frac{2}{\pi^2} r^2.
\end{equation*}

By basic properties of the trace, we have
\begin{equation*}
    \Tr((I-\cos \phi) P^T U \Sigma U^T P) \geq \lambda_{\min} (P^T U \Sigma U^T P) \Tr(I-\cos \phi) \geq \frac{2 \sigma_{\min}(C)}{\pi^2} \| \phi\|^2.
\end{equation*}
The last inequality completes the proof.

\end{proof}

We can combine Propositions \ref{prop:wqc} and \ref{prop:quad_growth} to a more compact form, which we call weak-quasi-strong-convexity (WQSC) \cite{alimisis2024characterization}. Similar versions of this property have also appeared in \cite{necoara2019linear} (Definition 1) and \cite{karimi2016linear} (Appendix A). WQSC looks a lot like strong convexity but it is much more general. Interestingly, the role of a strong convexity constant $\mu$ is played by a multiple of $\sigma_{\min}(C)$. That is to say, the further away from being singular $C$ is, the stronger this property becomes. If $C$ is singular, the derived inequality reduces to weak-quasi-convexity (Proposition \ref{prop:wqc}, but with slightly weaker parameters).

\begin{Proposition}[Weak-quasi-strong-convexity]
\label{prop:wqsc}

For any $X$ satisfying the properties of Propositions \emph{\ref{prop:wqc}, \ref{prop:quad_growth}}, $f$ satisfies the following inequality:
\begin{equation*}
    f(X)-f^* \leq \frac{1}{a(X)} \langle \textnormal{grad}f(X), - \log_X(X_*) \rangle - \frac{\mu}{2} \textnormal{dist}^2(X,X_*)
\end{equation*}
with $a(X):=\frac{1+\cos(|r|_{max})}{4}$ and $\mu:=\frac{4 \sigma_{min}(C)}{\pi^2}$. $|r|_{max}<\pi$ is the largest rotation in absolute value induced by the orthogonal matrix $X^T X_*$.
    
\end{Proposition}

\begin{proof}
For the specific choices of $a(X)$ and $\mu$, we have
\begin{equation*}
    \frac{\mu}{2} \textnormal{dist}^2(X,X_*) \leq f(X)-f_* \leq \frac{1}{2 a(X)} \langle \textnormal{grad}f(X), -\log_X(X_*) \rangle.
\end{equation*}
The left inequality is derived by Proposition \ref{prop:quad_growth} and the right one by Proposition \ref{prop:wqc}.

Now, again by Proposition \ref{prop:wqc}, we have
\begin{align*}
    f(X)-f_*  &\leq \frac{1}{2 a(X)} \langle \textnormal{grad}f(X), -\log_X(X_*) \rangle+\frac{\mu}{2} \textnormal{dist}^2(X,X_*)-\frac{\mu}{2} \textnormal{dist}^2(X,X_*) \\ & \leq \frac{1}{a(X)} \langle \textnormal{grad}f(X), -\log_X(X_*) \rangle -\frac{\mu}{2} \textnormal{dist}^2(X,X_*)
\end{align*}
by substituting the previous inequality. 
\end{proof}

We now analyze a property for $f$ usually referred as smoothness.

\begin{Definition}
 A function $f: \mathbb{O}(n) \rightarrow \mathbb{R} $ is called $L$-smooth, if for all $X,Y$ in the same connected component of $\mathbb{O}(n)$, we have
 \begin{equation}
 \label{eq:L_smoothness_full}
     \| \textnormal{grad}f(X) - \Gamma_Y^X \textnormal{grad}f(Y) \| \leq L \textnormal{dist}(X,Y),
 \end{equation}
 where $\Gamma_Y^X$ is the parallel transport between $X$ and $Y$.
\end{Definition}

Since $f$ is twice differentiable, the previous property is equivalent with the eigenvalues of its Riemannian Hessian being upper bounded in absolute value uniformly by $L$. Using a standard Taylor expansion (see \cite{absil2009optimization}, Theorem 7.1.2), we easily see that this implies
\begin{equation}
\label{eq:smoothness_1}
    f(Y)-f(X) \leq \langle \textnormal{grad}f(X),\log_X(Y) \rangle + \frac{L}{2} \textnormal{dist}^2(X,Y),
\end{equation}
for all $X,Y \in \mathbb{O}(n)$ that are connected by a unique geodesic.

A weaker, but useful, inequality comes by setting $Y=\exp_X\left(-\frac{1}{L} \textnormal{grad}f(X) \right)$ (see Proposition \ref{lem:GD convergence 1 step} for a proof that such $Y$ is uniquely defined). Then, we get
\begin{equation*}
    f(Y) \leq f(X) - \frac{1}{2L} \|\textnormal{grad}f(X) \|^2 \end{equation*}
    and since $f^* \leq f(Y)$,
\begin{equation}
\label{eq:smoothness_2}
    f(X)-f^* \geq \frac{1}{2L} \|\textnormal{grad}f(X)\|^2, \hspace{1mm} \text{for all} \hspace{1mm} X \in \mathbb{O}(n).
\end{equation}

\begin{Proposition}[Smoothness]
\label{prop:smoothness}
    $f$ is geodesically $\sigma_{\max}(C)-$smooth.
\end{Proposition}

\begin{proof}
    It suffices to show that the eigenvalues of the Riemannian Hessian $\textnormal{Hess}f(X)$ are in absolute value upper bounded by $\sigma_{\max}(C)$ for all $X \in \mathbb{O}(n)$. For our computations, we follow the exposition in \cite{boumal2023}:   
\begin{equation*}
    \langle \dot X , \textnormal{Hess}f(X)[\dot X] \rangle = \Tr(\dot X^T \dot X \textnormal{skew} (-X^T C^T )) + \Tr(\dot X^T X  \textnormal{skew} (- \dot X^T C^T )) .
\end{equation*}

The first term is $0$ as the trace of the product of a symmetric and skew-symmetric matrix. The second term becomes
\begin{equation*}
\Tr(\dot X^T X  \textnormal{skew} (- \dot X^T C^T ))  = \frac{1}{2} \Tr(\dot X^T X C \dot X- \dot X^T X \dot X^T C^T).
\end{equation*}
Substituting $\dot X^T X = \Omega^T$, we get
\begin{align*}
    & \frac{1}{2} \Tr(\dot X^T X C X- \dot X^T X \dot X^T C^T) = \frac{1}{2} \Tr(\Omega^T C \dot X - \Omega^T \dot X^T C^T ) = \\ & \frac{1}{2} \Tr(\Omega^T C \dot X + \Omega \dot X^T C^T) = \Tr(\Omega^T C \dot X) = \Tr(\Omega^T C X \Omega) = \Tr(CX \Omega \Omega^T).
\end{align*}
The last expression features the trace of the product of the matrix $CX$ with the symmetric and positive semi-definite matrix $\Omega \Omega^T$. By basic facts in linear algebra (Von Neumann's trace inequality), we can upper bound the absolute value of this expression by $\sigma_{\max}(CX) \Tr(\Omega \Omega^T)$. Since $X$ is orthogonal, we have that $\sigma_{\max}(CX) = \sigma_{\max}(C)$. Also $\Tr(\Omega \Omega^T) = \Tr(\dot X \dot X^T) = \| \dot X \|^2$. Putting it all together, we get
\begin{equation*}
    | \langle \dot X , \textnormal{Hess}f(X)[\dot X] \rangle | \leq \sigma_{\max}(C) \| \dot X \|^2
\end{equation*}
and the desired result follows.
\end{proof}

As it is customary, we denote 
\begin{equation*}
    L:=\sigma_{\max}(C).
\end{equation*}

\begin{remark}
   Since we have computed the quantity $\langle \dot X , \textnormal{Hess}f(X)[\dot X] \rangle$, it is a good point to verify that $f$ is indeed non-convex, even inside the same connected component. To see that, take $ C = \begin{bmatrix}
1 & 0 \\
0 & 2
\end{bmatrix}. $ It is easy to see that in this case, the minimizer of $f$ is the identity matrix $X_*=\begin{bmatrix}
1 & 0 \\
0 & 1
\end{bmatrix}$. Consider a matrix, which is almost equal to $\begin{bmatrix}
-1 & 0 \\
0 & -1
\end{bmatrix}$, but not exactly, take for example $X:=\begin{bmatrix}
\cos \theta & \sin \theta \\
-\sin \theta & \cos \theta
\end{bmatrix}$, with any $\theta \in (\pi/2,\pi)$. The eigenvalues of $X^T X_*$ are $e^{i \theta}$ and $e^{-i \theta}$ and, since $|\theta| < \pi$, $X$ and $X^*$ are connected by a unique geodesic. The Riemannian Hessian at $X$ satisfies
\begin{equation*}
    \langle \dot X , \textnormal{Hess}f(X)[\dot X] \rangle = \Tr(CX \Omega \Omega^T), 
\end{equation*}
where $\dot X = X \Omega$. We have that
$ CX = 
\begin{bmatrix}
\cos \theta & \sin \theta \\
-\sin \theta & 2 \cos \theta
\end{bmatrix}
$
and set the symmetric and positive semidefinite matrix $\Omega \Omega^T$ to be $
\begin{bmatrix}
\alpha & \beta \\
\beta & \gamma
\end{bmatrix},
$
where $\alpha$ and $\gamma$ are nonnegative and at least one is strictly positive.
\newline
The diagonal entries of the matrix $CX \Omega \Omega^T$ are $\alpha \cos \theta - \beta \sin \theta$ and $\beta \sin \theta + 2 \gamma \cos \theta$. Thus, 
\begin{equation*}
 \Tr(CX \Omega \Omega^T) = \underbrace{(\alpha+2\gamma)}_{>0} \underbrace{\cos \theta}_{<0}  
\end{equation*}
and we have found a point $X$ where the Riemannian Hessian has negative eigenvalues (it is actually negative definite).
\end{remark}

We conclude this section with a small technical lemma that allows us to show that gradient descent with a properly chosen step size is well-defined in the sense that the direction used for update belong in the injectivity domain (\ref{inj_domain}).

\begin{lemma}
\label{le:ell_2_norm}
    The Riemannian gradient of $f$ evaluated at $X$ is of the form $X \Omega$, for a skew-symmetric matrix $\Omega$ with
    \begin{equation*}
        \|\Omega\|_2 \leq \sigma_{\max}(C).
    \end{equation*}
\end{lemma}
\begin{proof}
   The Riemannian gradient of $f$ at $X$ is
   \begin{equation*}
       \textnormal{grad}f(X)=X \textnormal{skew}(X^T C^T),
   \end{equation*}
   thus $\Omega$ is taken as $\textnormal{skew}(X^T C^T)$.
   By the sub-additivity of the spectral norm and its invariance under multiplication with orthogonal matrices, we have
   \begin{align*}
       \|\Omega\|_2 = \|\textnormal{skew}(X^T C^T) \|_2 & = \leq \frac{\|X^T C^T\|_2 + \| CX\|_2}{2} =   \frac{\|C^T\|_2 + \| C\|_2}{2}.
   \end{align*}
   This gives the desired result.
\end{proof}

\section{Convergence of Riemannian gradient descent}

Riemannian gradient descent applied to a function $f:\mathbb{O}(n) \rightarrow \mathbb{R}$ reads as
\begin{equation}
\label{eq:RGD}
    X_{t+1} = \exp_{X_t} (-\eta_t \textnormal{grad}f(X_t)),
\end{equation}
with $\eta_t>0$ being the step size.

The results of Section \ref{sec:main_theory} guarantee a local (non-asymptotic) linear convergence rate for Riemannian gradient descent on $f$ in the case that $C$ is invertible, if ran with a properly chosen step size and the initial guess $X_0$ is sufficiently close to the optimum. 

\begin{Proposition}\label{lem:GD convergence 1 step}
Let $X_t$ and $X_*$ be such that the largest rotation $|r|_{\max}$ induced by the orthogonal matrix $X_t^T X_*$ satisfies $|r|_{\max}<\pi$.
Then, iteration~\eqref{eq:RGD} with $0 \le \eta_t \leq a(X_t)/L$ satisfies
   \begin{equation*}
   \textnormal{dist}^2(X_{t+1},X_*)   \leq \left(1- \frac{4}{\pi^2} \sigma_{\min}(C) a(X_t) \eta_t \right)  \textnormal{dist}^2(X_t,X_*),
   \end{equation*}
with $a(X_t)$ defined as in Proposition \ref{prop:wqsc}.
\end{Proposition}

\begin{proof} We start by showing that iteration \ref{eq:RGD} is well-defined.
By the assumption $|r|_{\max}<\pi$, we get that $0< a(X_t) = \frac{1+\cos(|r|_{\max}))}{4} \leq \frac{1}{2}$. 
By Lemma \ref{le:ell_2_norm}, the tangent vector $\eta_t \textnormal{grad}f(X_t)$ that is used to update iteration \ref{eq:RGD} can be written as $X \Omega$, with $\|\Omega\|_2 \leq \eta_t \sigma_{max}(C) $. By the definition of $\eta_t$, we have that
\begin{equation*}
    \|\Omega\|_2 \leq \frac{a(X_t)}{L} \sigma_{\max}(C) = \frac{a(X_t)}{\sigma_{\max}(C)} \sigma_{\max}(C) \leq \frac{1}{2}.
\end{equation*}
Thus, $\eta_t \textnormal{grad}f(X_t)$ is inside the injectivity domain (\ref{inj_domain}) and, as a consequence, iteration (\ref{eq:RGD}) is well-defined.

We can now apply Proposition~\ref{prop:law_of_cosines} to obtain
\begin{align}\label{eq:dist_with_sigma}
        \textnormal{dist}^2(X_{t+1},X_*)  &\leq \|-\eta_t \textnormal{grad} f(X_t)-\log_{X_t}(X_*) \|^2   \notag \\ 
        & = \eta_t^2 \|\textnormal{grad} f(X_t)\|^2 + \textnormal{dist}^2(X_t, X_*) +2 \eta_t \, \sigma \
    \end{align}
    with
    \[
    \sigma := \langle \textnormal{grad} f(X_t),\log_{X_t}(X_*) \rangle.
    \]
    Proposition~\ref{prop:wqsc} and equation (\ref{eq:smoothness_2}) give
    \begin{align*}
        \frac{\sigma}{a(X_t)} &\leq f^*-f(X_t)-\frac{2 \sigma_{\min}(C)}{\pi^2}\textnormal{dist}^2(X_t,X_*) \\
        &\leq -\frac{1}{2 L} \| \textnormal{grad} f(X_t) \|^2- \frac{2 \sigma_{\min}(C)}{\pi^2} \textnormal{dist}^2(X_t,X_*).
    \end{align*}
    Multiplying by $2  a(X_t)\, \eta_t$ and using $\eta_t \leq a(X_t) / L$, we get
    \begin{align*}
        2 \eta_t \, \sigma &\leq -\frac{a(X_t) \, \eta_t }{L} \| \textnormal{grad} f(X_t) \|^2 - \frac{4 \sigma_{\min}(C)}{\pi^2} a(X_t) \, \eta_t \, \textnormal{dist}^2(X_t,X_*) \\ & \leq -\eta_t^2 \| \textnormal{grad} f(X_t) \|^2 - \frac{4 \sigma_{\min}(C)}{\pi^2} a(X_t)\, \eta_t \, \textnormal{dist}^2(X_t, X_*).
         \end{align*}
         Substituting into~\eqref{eq:dist_with_sigma}, we obtain the desired result.
\end{proof}

\begin{theorem}[Convergence of RGD for the Procrustes problem]
\label{thm:GD_conv}
    Let $C$ be invertible ($\sigma_{\min}(C) > 0$) and $X_*$ the (unique) minimizer of $f$. Then, Riemannian gradient descent (\ref{eq:RGD}) in the orthogonal group, starting by a point $X_0 \in \mathbb{O}(n)$ such that \begin{equation*}
        \textnormal{dist}(X_0,X_*) < \pi,
    \end{equation*} and ran with fixed step size $$\eta_t \equiv \eta \leq \frac{1+\cos(\textnormal{dist}(X_0,X_*))}{4 \sigma_{\max}(C)},$$ produces iterates $X_t$ that satisfy
    \begin{equation*}
        \textnormal{dist}^2(X_t,X_*) \leq \left( 1-\frac{1}{\pi^2}(1+\cos(\textnormal{dist}(X_0,X_*))) \sigma_{\min}(C) \eta \right)^t \textnormal{dist}^2(X_0,X_*).
    \end{equation*}
    
\end{theorem}

\begin{proof}
We do the proof by induction. 

For $t=0$, the inequality is trivially true.

We now assume that the inequality is true for $t$ and we wish to show that it is true also for $t+1$. 

Since $\textnormal{dist}(X_t,X_*) \leq \textnormal{dist}(X_0,X_*)$, we also get that the largest possible rotation $|r(X_t, X_*)|_{\max}$ induced by $X_t^T X_*$ satisfies
\[
 |r(X_t, X_*)|_{\max} \leq \sqrt{\sum_{i=1}^n r_i (X_t, X_*)^2} = \textnormal{dist}(X_t, X_*) \leq \textnormal{dist}(X_0,X_*),
\]
where $r_i(X_t,X_*)$ are the rotations induced by the matrix $X_t^T X_*$. The first equality comes from equation (\ref{eq:Riem_distance}).

By the definition of $a(X_t)$ in Proposition (\ref{prop:wqsc}), we have
\begin{equation*}
    a(X_t) = \frac{1+\cos(|r(X_t, X_*)|_{\max})}{4} \geq \frac{1+\cos(\textnormal{dist}(X_0,X_*))}{4},
\end{equation*}
thus $\eta \leq a(X_t)/L$.

Since $\eta$ satisfies the previous bound, the outcome of Proposition \ref{lem:GD convergence 1 step} holds, and combining it with the induction hypothesis, we get
\begin{align*}
   & \textnormal{dist}^2(X_{t+1},X_*) \leq \left(1- \frac{4}{\pi^2} \sigma_{\min}(C) a(X_t) \eta \right)  \textnormal{dist}^2(X_t,X_*) \leq \\ & \left( 1-\frac{1}{\pi^2}(1+\cos(\textnormal{dist}(X_0,X_*))) \sigma_{\min}(C) \eta \right) \textnormal{dist}^2(X_t,X_*) \leq \\ & \left( 1-\frac{1}{\pi^2}(1+\cos(\textnormal{dist}(X_0,X_*))) \sigma_{\min}(C) \eta \right)^{t+1} \textnormal{dist}^2(X_0,X_*).
\end{align*}

This concludes the induction.

\end{proof}

\begin{remark}
   If $C$ is singular, then the previous theorem only states that the distances of the iterates of gradient descent to the set of optima do not increase. In that case we can still prove an algebraic convergence rate for the function values of Riemannian gradient descent based only on weak-quasi-convexity. 
\end{remark}

\begin{remark}
 The assumption $\textnormal{dist}(X_0,X_*)<\pi$ allows to bound globally $|r(X_t,X_*)|_{\max}$ from above by $\textnormal{dist}(X_0,X_*)$ and as a result keep the quantity $1+\cos(|r(X_t,X_*)|_{\max})$ far away from $0$ over the course of gradient descent. Intuitively, it does not allow the algorithm to go too close to non-optimal critical points. Gradient descent would not stick to non-optimal critical points, but it would probably slow down a lot.   
\end{remark}

We now show an algebraic convergence rate for gradient descent that covers also the case that $C$ is singular.

\begin{theorem}
\label{thm:GD_conv_no_gap1}
    Gradient descent applied to $f$ for any square non-zero matrix $C$, starting from $X_0 \in \mathbb{O}(n)$  such that
    \begin{equation*}
        \textnormal{dist}(X_0,X_*) < \pi
    \end{equation*}
    and with fixed step size
    $$\eta \leq \frac{1+\cos(\textnormal{dist}(X_0,X_*))}{4 \sigma_{\max}(C)},$$ produces iterates $X_t$ that satisfy
    \begin{equation*}
    f(X_t) - f^* \leq \frac{2L+\frac{1}{\eta}}{ (1+\cos(\textnormal{dist}(X_0,X_*)))t+4}  \textnormal{dist}^2(X_0,X_*)=\mathcal{O}\left(\frac{1}{t} \right).  
\end{equation*}
    
\end{theorem}

\begin{proof}
Since we still satisfy all the hypotheses of Theorem \ref{thm:GD_conv}, we know that for all $t\geq 0$ it holds
$\textnormal{dist}(X_t,X_*) \leq 
\textnormal{dist}(X_0,X_*) < \pi$. This implies that
\begin{equation*}
   a(X_t) 
   \geq \frac{1+\cos(\textnormal{dist}(X_0,X_*))}{4} > 0,
\end{equation*}
which implies that the function $f$ is weakly-quasi-convex (Proposition \ref{prop:wqc}) at every $X_t$ such that:
\begin{equation*}
        \langle \textnormal{grad}f(X_t), -\log_{X}(X_*) \rangle \geq \frac{1}{2} (1+\cos(\textnormal{dist}(X_0,X_*))) (f(X_t)-f_*).
    \end{equation*}

Denoting $C_0:=\frac{1+\cos(\textnormal{dist}(X_0,X_*))}{4}$ and $\Delta_t := f(X_t) - f^*$, we can write
\begin{equation}\label{eq:weak-quasi-conv_with_Delta_t}
2 C_0  \Delta_t \leq \langle \textnormal{grad} f(X_t) , - \log_{X_t} (X_*) \rangle.
\end{equation}

Similar to the proof of Theorem \ref{thm:GD_conv}, by the hypothesis on the step size $\eta_t$, Lemma~\ref{lem:GD convergence 1 step} shows that $-\eta_t X_{t+1}$ is in the injectivity domain of $\exp$ at $X_t$. 
Hence, by the definition of Riemannian gradient descent, we have
\begin{equation}\label{eq:Log_SD}
    \log_{X_t} (X_{t+1})=-\eta \textnormal{grad} f(X_t).
\end{equation}
In addition, the smoothness property of $f$ (equation \ref{eq:smoothness_1}) gives
\begin{equation*}
    \Delta_{t+1}-\Delta_t \leq \langle \textnormal{grad} f(X_t), \log_{X_t} (X_{t+1}) \rangle+\frac{L}{2} \textnormal{dist}^2(X_t,X_{t+1}).
\end{equation*}
Substituting~\eqref{eq:Log_SD}, we obtain
\begin{equation}\label{eq:conv_SD_delta_zero_diff_Delta}
    \Delta_{t+1}-\Delta_t \leq \left(-\eta +\frac{L}{2} \eta^2 \right) \| \textnormal{grad} f(X_t)\|^2 \leq 0.
\end{equation}

By Proposition \ref{prop:law_of_cosines}, we have
\begin{equation*}
    \textnormal{dist}^2(X_{t+1}, X_*) \leq \textnormal{dist}^2(X_t, X_{t+1})+ \textnormal{dist}^2(X_t, X_*)-2 \langle \log_{X_t} (X_{t+1}), \log_{X_t} (X_*) \rangle.
\end{equation*}
Substituting~\eqref{eq:Log_SD} into the above and rearranging terms gives
\begin{equation*}
    2 \eta \langle \textnormal{grad} f(X_t) , - \log_{X_t} (X_*) \rangle \leq \textnormal{dist}^2(X_t, X_*)-\textnormal{dist}^2(X_{t+1}, X_*)+ \eta^2 \| \textnormal{grad} f(X_t) \|^2.
\end{equation*}
Combining with~\eqref{eq:weak-quasi-conv_with_Delta_t}, we  get
\begin{equation}\label{eq:bound_Delta_t}
    \Delta_t \leq \frac{1}{4 C_0 \eta} ( \textnormal{dist}^2(X_t, X_*)-\textnormal{dist}^2(X_{t+1}, X_*)) + \frac{\eta}{4 C_0} \| \textnormal{grad} f(X_t) \|^2.
\end{equation}
Now multiplying \eqref{eq:conv_SD_delta_zero_diff_Delta} by $\frac{1}{C_0}$ and summing with~\eqref{eq:bound_Delta_t} gives
\begin{multline}\label{eq:diff_Delta_intermediate}
    \frac{1}{C_0} \Delta_{t+1} - \left( \frac{1}{ C_0} - 1  \right) \Delta_t  \leq  \frac{1}{4 C_0 \eta} ( \textnormal{dist}^2(X_t, X_*)-\textnormal{dist}^2(X_{t+1}, X_*)) \\  
    +\frac{1}{C_0} \left( -\eta + \frac{L}{2}\eta^ 2 + \frac{\eta}{4} \right)
    \| \textnormal{grad} f(X_t) \|^2.
\end{multline}
By assumption, we have $\eta \leq C_0/L$, where $0 < C_0= (1+\cos(\textnormal{dist}(X_0, X_*)))/4 \leq 1$ and $L > 0$. Since
\begin{equation*}
    \frac{\eta}{C_0} \left( -1 + \frac{L}{2} \eta + \frac{1}{4}  \right) \leq \frac{\eta}{C_0} \left(\frac{C_0}{2} -\frac{3}{4} \right) \leq - \frac{1}{4} \frac{\eta}{C_0}< 0.
\end{equation*}
Inequality~\eqref{eq:diff_Delta_intermediate} can be simplified to
\begin{equation*}
     \frac{1}{C_0} \Delta_{t+1} - \left( \frac{1}{ C_0} - 1  \right) \Delta_t \leq \frac{1}{4 C_0 \eta} ( \textnormal{dist}^2(X_t, X_*)-\textnormal{dist}^2(X_{t+1}, X_*)).
\end{equation*}
Summing from $0$ to $t-1$ gives
\[
    \frac{1}{C_0} \Delta_t + \sum_{s=1}^{t-1} \Delta_s - \left( \frac{1}{C_0} -1  \right) \Delta_0 \leq  \frac{1}{4 C_0 \eta} \left( \textnormal{dist}^2(X_0, X_*) - \textnormal{dist}^2(X_t, X_*) \right).
\]
From the smoothness property (\ref{eq:smoothness_1}) with $Y \rightsquigarrow X$ and $X \rightsquigarrow X_*$, we get 
\[
 \Delta_0 \leq \frac{L}{2} \textnormal{dist}^2(X_0, X_*).
\]
Combining these two inequalities leads to
\begin{align*}
    \frac{1}{C_0} \Delta_t + \sum_{s=0}^{t-1} \Delta_s & \leq \frac{1}{C_0}  \Delta_0 + \frac{1}{4 C_0 \eta} \textnormal{dist}^2(X_0, X_*) \\ 
    & \leq \frac{1}{2C_0} \left(L +\frac{1}{2 \eta}\right) \textnormal{dist}^2(X_0, X_*).
\end{align*}
Since \eqref{eq:conv_SD_delta_zero_diff_Delta} holds for all $t \geq 0$, it also implies $\Delta_t \leq \Delta_s$ for all $0 \leq s \leq t$. Substituting
\[
 t \Delta_t \leq \sum_{s=0}^{t-1} \Delta_s
\]
into the inequality from above, we obtain 
\begin{equation*}
    \Delta_t \leq \frac{1}{2C_0} \frac{L+\frac{1}{2 \eta}}{\frac{1}{C_0}+t} \textnormal{dist}^2(X_0, X_*) = \frac{L+\frac{1}{2\eta}}{2(C_0 t+1)} \textnormal{dist}^2(X_0, X_*),
\end{equation*}
we obtain the desired result.
\end{proof}

\section{Discussion}
In this work, we showed that the (non-convex) problem of computing the polar factor of a square matrix satisfies a convexity-like structure in the orthogonal group. This structure allows the analysis of classic algorithms like gradient descent, which are not competitive theoretically and empirically compared to the state-of-the-art. An interesting direction for future work could be to explore cases of "noisy" polar decomposition that cannot be solved using the standard algorithms. An example is robust polar decomposition in the sense of solving the problem
\begin{equation*}
    \min_{X \in \mathbb{O}(n)} \max_{C \in \mathbb{R}^{n \times n}} \left(-\Tr(C X) - \beta \sum_{i=1}^s \|C-C_i\|^2 \right),
\end{equation*}
where $\lbrace C_i \rbrace_{i=1}^s$ is a set of independent observations for $C$ and $\beta>0$ is a regularizer. To the best of our knowledge, traditional linear algebra techniques cannot be applied to such problem. A more viable approach would be min-max optimization (for instance gradient descent ascent), for which our theory could be of value.
\newline
In general, we think that the direction of discovering convexity structures for non-convex but still tractable problems in linear algebra is promising (as highlighted also by the work \cite{alimisis2024geodesic}) and can shed light to many aspects of algorithmic computation in this field. Regarding polar decomposition, it is still of value to see whether a WQSC structure holds for the case that the matrix $C$ is rectangular (not square). In this case, we have an optimization problem over the Stiefel manifold. We predict  that this problem will be much more involved, due to the more involved Riemannian structure of the Stiefel manifold.

\bibliographystyle{plain} 
\bibliography{refs}

\end{document}